\documentclass{article}

\usepackage[margin=1in]{geometry}

\usepackage{graphicx}
\graphicspath{{./figs/}}

\usepackage{amsmath}
\usepackage{authblk}

\makeatletter
\DeclareMathOperator*{\infd}{inf\vphantom{\operator@font p}}
\DeclareMathOperator*{\supx}{su\smash{\operator@font p}}
\makeatother

\usepackage{amsthm}
\usepackage{amssymb}
\usepackage{multicol}
\usepackage{marginnote}
\usepackage{todonotes}
\usepackage{tikz-cd}

\newcommand{\R}{\mathbb{R}}

\newcommand{\ds}{\displaystyle}

\newcommand\Dg{\mathrm{Dg}}

\newcommand{\DgPD}[1]{\Dg{ #1}}

\newtheorem{thm}{Theorem}

\newtheorem{lem}[thm]{Lemma}

\newtheorem{cor}[thm]{Corollary}
\newtheorem{defn}[thm]{Definition}

\newcommand {\mm}[1] {\ifmmode{#1}\else{\mbox{\(#1\)}}\fi}
\newcommand{\Rspace}        {\mm{{\mathbb R}}}
\newcommand{\Xspace}        {\mm{{\mathbb X}}}
\newcommand{\Yspace}        {\mm{{\mathbb Y}}}
\newcommand{\Kspace}        {\mm{{\mathbb K}}}
\newcommand{\Vspace}        {\mm{{\mathbb V}}}
\newcommand{\Ispace}        {\mm{{\mathbb I}}}

\newcommand{\Uspace}        {\mm{{\mathbb U}}}

\newcommand{\Xbf}        {\mm{{\textbf X}}}
\newcommand{\Ybf}        {\mm{{\textbf Y}}}

\newcommand{\Hgroup}        {\mm{{\mathsf H}}}

\newcommand{\id}        {\mm{\mathrm{id}}}

\begin{document}

\title{Local Versus Global Distances for Zigzag Persistence Modules}

\author[1]{Ellen Gasparovic\thanks{gasparoe@union.edu}}
\author[2]{Maria Gommel\thanks{maria-gommel@uiowa.edu}}
\author[3]{Emilie Purvine\thanks{emilie.purvine@pnnl.gov}}
\author[4]{Radmila Sazdanovic\thanks{rsazdanovic@math.ncsu.edu}}
\author[5]{Bei Wang\thanks{beiwang@sci.utah.edu}}
\author[6]{Yusu Wang\thanks{yusu@cse.ohio-state.edu}}
\author[7]{Lori Ziegelmeier\thanks{lziegel1@macalester.edu}}
\affil[1]{Union College, Schenectady, NY}
\affil[2]{University of Iowa, Iowa City, IA}
\affil[3]{Pacific Northwest National Laboratory, Seattle, WA}
\affil[4]{North Carolina State University, Raleigh, NC}
\affil[5]{University of Utah, Salt Lake City, UT}
\affil[6]{Ohio State University, Columbus, OH}
\affil[7]{Macalester College, Saint Paul, MN}

\maketitle

\begin{abstract}
This short note establishes
explicit and broadly applicable relationships between persistence-based distances computed locally and globally. In particular, we show that the bottleneck distance between two zigzag persistence modules restricted to an interval is always bounded above by the distance between the unrestricted versions.
While this result is not surprising, it could have different practical implications. We give two related applications for metric graph distances, as well as an extension for the matching distance between multiparameter persistence modules.

\end{abstract}

\section{Introduction}
\label{sec:introduction}

\paragraph{Persistence modules and zigzag persistence}
The theory of persistence modules is at the core of topological data analysis. 
The theory begins with the study of 1-parameter persistence modules over $\Rspace$-valued functions. 
In the ordinary setting, given a diagram of topological spaces connected via inclusion maps,
\[\Xspace_1 \to \Xspace_2 \to \cdots \to \Xspace_n,\] we apply the $p$-dimensional homology functor $\Hgroup_p$ with coefficients in a field $\Kspace$ to obtain a diagram of vector spaces with linear maps, 
\[\Vspace_1 \to \Vspace_2 \to \cdots \to \Vspace_n,\] where $\Vspace_i = \Hgroup_p(\Xspace_i; \Kspace)$. 
Such a diagram is called a 1-parameter persistence module~\cite{CarlssonSilva2010}. 
Various persistence modules generalizing the 1-parameter setting have been studied in the literature, 
including generalized~\cite{BubenikSilvaScott2015} (i.e.,~over posets), zigzag~\cite{BotnanLesnick2018,CarlssonSilva2010} persistence modules, and multiparameter~\cite{Lesnick2015} (i.e.,~over $\Rspace^d$-valued functions); see~\cite{BubenikVergili2018} for a description of their relationships. 

We focus on zigzag persistence modules, which, in a nutshell, allow arrows to point in either direction~\cite{CarlssonSilva2010}. 
Given a diagram of topological spaces connected by inclusion maps,
\[\Xspace_1 \leftrightarrow \Xspace_2 \leftrightarrow \cdots \leftrightarrow  \Xspace_n,\] we apply the homology functor as usual to obtain a sequence of vector spaces and linear maps, \[\Vspace_1 \leftrightarrow \Vspace_2 \leftrightarrow \cdots \leftrightarrow \Vspace_n,\]  where each $\leftrightarrow$ represents either a forward or a backward map. 
Zigzag persistence modules generalize the classic 1-parameter setting and handle several situations which are not covered by the classic theory. 
Linearity allows a zigzag persistence module (similar to a 1-parameter persistence module) to be uniquely decomposed  into elementary  pieces (called indecomposable modules) which are intervals. The information encoded by these intervals can be combinatorially represented by the persistence diagram~\cite{EdelsbrunnerMorozov2017}.
In the case of multiparameter persistence, such indecomposable modules are complex and no longer intervals. 
We are interested in zigzag persistence as it involves the most general type of linear module that still gives rise to classic persistence diagrams. 
Furthermore, a zigzag persistence module can be used to compute ordinary persistent homology with good space efficiency (see Section~\ref{sec:bottleneck} for details). 

To measure the distance between persistence modules, the notion of \emph{interleaving distance} has been employed~\cite{ChazalCohen-SteinerGlisse2009} which captures the proximity between persistence modules.  
For 1-parameter persistence modules, it has been shown that the interleaving distance is equal to the well-known \emph{bottleneck distance}~\cite{Cohen-SteinerEdelsbrunnerHarer2007} between the persistence diagrams of the corresponding persistence modules~\cite{Lesnick2015}. 
In this paper, we prove a straightforward inequality involving the bottleneck distance between persistence diagrams \cite{Cohen-SteinerEdelsbrunnerHarer2007} that is useful for data analysis. 

\paragraph{Global versus local perspectives on persistence} 
We are motivated by the study of persistence modules from both global and local perspectives. 
A persistence module provides a global 
description of a complex dataset, and we are interested in quantifying the amount of information that is preserved when restricted to local neighborhoods or intervals. 

For a first example, consider the question of determining or approximating graph motif counts.
A graph motif is a subgraph on a small number of vertices contained within a larger, more complex graph. 
Graph motifs have proven useful for characterizing networks in domains like biology \cite{Shen-OrrMiloMangan2002} and cyber security \cite{HarshawBridgesLannacone2016}.
The standard problem of counting the number of small motifs or patterns within a graph is equivalent to the subgraph isomorphism problem, which is NP-complete. 
Since restricted persistence modules reveal information about the local structure of a space, we posit that the restricted modules for a metric graph (see Section \ref{sec:applications}) can be used similarly to how graph motifs are currently used, e.g., as inputs to classification algorithms or anomaly detection algorithms in time-varying data  \cite{HarshawBridgesLannacone2016,Milo824}. 

For a second example, consider persistent local homology, which studies a multi-scale notion of homology within a local neighborhood of the data relative to its boundary. 
It has applications in road network analysis~\cite{AhmedFasyWenk2014}, local dimension estimation~\cite{DeyFanWang2014}, data visualization~\cite{WangSummaPascucci2011}, graph reconstruction~\cite{ChernovKurlin2013,AanjaneyaChazalChen2012}, clustering and stratification learning~\cite{BendichMukherjeeWang2012, BendichCohen-SteinerEdelsbrunner2007}. 
Furthermore, persistent local homology extracts local geometric and topological information in data, which can be used as input to machine learning algorithms~\cite{BendichGasparovicHarer2015}.


\paragraph{Our contributions}
We show that the bottleneck distance between two zigzag persistence modules restricted over an interval of parameter values is always bounded by the distance between the unrestricted versions (Theorem~\ref{thm:changeradius}) and state a corollary in the case of \emph{level set zigzag persistence} (Corollary~\ref{cor:levelsetzigzag}). We also establish two results involving distance inequalities in the special case of metric graphs (Corollary~\ref{cor:changebasepoint} and Corollary~\ref{cor:localpd}) and point out how our results can be extended to multiparameter persistence modules. 

The results in this short paper have the potential for many diverse applications across different settings. 
For instance, 
if one wishes to compare the persistence profiles of two very large data sets but finds that it is prohibitively computationally expensive, one has the option to compute a restricted version of the bottleneck distance as an approximation to the global distance. 
As the interval size increases, the bottleneck distance between the restricted versions approaches the distance for the global versions.

Relatedly, it may be the case that two long zigzag sequences need to be compared on a local scale. The question may be: are there any local differences between the two zigzag sequences? One could do many local comparisons to answer this question. However, our result means that a small global distance between the two zigzag persistence diagrams implies small local distances. To save computation one could compute the global distance as a first step. Local distances only need to be computed if the global distance is large.

 Restricted persistence modules may be helpful for analyzing time-varying systems. Given data $\Xspace_t$ at time $t$ (e.g., a graph, function, or point cloud), a zigzag persistence module can be constructed for the sequence \[\Xspace_1 \rightarrow \Xspace_1 \cup \Xspace_2 \leftarrow \Xspace_2 \rightarrow \Xspace_2 \cup \Xspace_3 \leftarrow \cdots\]
where all of the maps are inclusion maps. A subinterval of this sequence corresponds to a time interval contained within the larger sample. Given two long time intervals, one could either compare them in full or compare smaller windows. Our result shows that the local differences contained in small time intervals are not ``washed out'' as one moves to larger intervals.

The rest of the paper is organized as follows. In Section \ref{sec:background}, we recall the necessary concepts for zigzag persistence. Our main theorem is contained in Section \ref{sec:bottleneck}, and we consider applications of the theorem in the metric graph setting and for multi-parameter persistence in Section~\ref{sec:applications}. We conclude with a discussion of future work 
in Section \ref{sec:conclusion}.

\section{Brief Background and Definitions}
\label{sec:background}

Our treatment of zigzag persistence is brief; for more details, see~\cite{CarlssonSilva2010} and~\cite{CarlssonSilvaMorozov2009}. 
A
\emph{zigzag diagram} of topological spaces $\Xspace_1,\Xspace_2,\ldots,\Xspace_n$ is a sequence
\[\Xspace_1 \leftrightarrow \Xspace_2 \leftrightarrow \cdots \leftrightarrow \Xspace_n\]
where each bidirectional arrow between two topological spaces represents a continuous function mapping either forwards or backwards. Applying the $p$-th homology functor with
coefficients in a field $\mathbb{K}$ yields a zigzag
diagram of vector spaces \[\Hgroup_p(\Xspace_1) \leftrightarrow \Hgroup_p(\Xspace_2) \leftrightarrow \cdots \leftrightarrow \Hgroup_p(\Xspace_n),\] known as a \emph{zigzag module}, denoted as $\Xbf$, 
from which \emph{zigzag persistence} may be computed. A zigzag module   decomposes into intervals $\Xbf \cong \displaystyle \bigoplus_{j \in J} \Ispace[b_j, d_j]$, where each $\Ispace[b_j, d_j]$ is defined as 
\[0 \longleftrightarrow \cdots \longleftrightarrow 0 \longleftrightarrow \mathbb{K} \longleftrightarrow \cdots \longleftrightarrow \mathbb{K} \longleftrightarrow 0 \cdots \longleftrightarrow 0 \]
with nonzero values in the range $[b_j,d_j]$. We will use $\DgPD{\Xbf}$ to denote the resulting \emph{persistence diagram} of a fixed homology dimension $p$. By Proposition 2.12 of \cite{CarlssonSilva2010}, restricting the module $\Xbf$ to the range $[r_1,r_2]$ (denoted $\Xbf[r_1,r_2]$) yields a decomposition as the direct sum of the intervals in $\Xbf$ restricted to $[r_1,r_2]$; that is, 
\begin{align}\label{eqn:restrictioin}
    \Xbf[r_1,r_2] \cong  \displaystyle \bigoplus_{j \in J} \Ispace([b_j, d_j] \cap [r_1,r_2]).
\end{align}

The \emph{bottleneck distance} between two persistence diagrams is equal to $\delta$ if there exists a matching between the points of the two diagrams (where points are allowed to be matched to diagonal elements) such that any pair of matched points are at distance at most $\delta$.
Formally, for a fixed homology dimension, 
the bottleneck distance is given by
\[
d_B(\Dg \Xbf, \Dg \Ybf) = \inf_{\mu}\sup_x ||x - \mu(x)||_{\infty},
\]
where $\mu$ ranges over all bijections between the two diagrams~\cite{EdelsbrunnerHarer2008}.


We conclude this section by defining a projection map that keeps track of the points in the global persistence diagram that disappear in the restricted version. 
The validity of the projection map in the following definition is guaranteed by Proposition 2.12 of \cite{CarlssonSilva2010} which leads to equation (\ref{eqn:restrictioin}). 

\begin{defn}\label{Res} 
Given $I=[r_1,r_2] \subset \Rspace$, we let $\DgPD{\Xbf^I}$ denote the \textbf{restriction of the persistence diagram $\DgPD{\Xbf}$ to the interval $I$} defined via the following projection map:

\noindent\parbox{0.45\textwidth}{
\vspace{2mm}
\begin{align*}
\Pi: \DgPD{\Xbf}  & \to \DgPD{\Xbf^I} \\
(b,d)  & \mapsto \begin{cases} (b,d) \text{ if } r_1 \leq b \leq d \leq r_2 \text{ (Case A)}\\
(b,r_2) \text{ if } r_1 \leq b \leq r_2 \leq  d \text{ (Case B)}\\
(r_1,d) \text{ if } b \leq r_1 \leq d \leq  r_2 \text{ (Case C)}\\
(r_1,r_2) \text{ if } b \leq r_1 \leq r_2 \leq  d \text{ (Case D)}\\
(b,b) \text{ if }  r_2 \leq b \text{ (Case E)}\\
(d,d) \text{ if }  d \leq r_1 \text{ (Case F)}\\
\end{cases}
\end{align*}}
\parbox{0.45\textwidth}{
\vspace{2mm}
\centering
\includegraphics[width=0.45\textwidth]{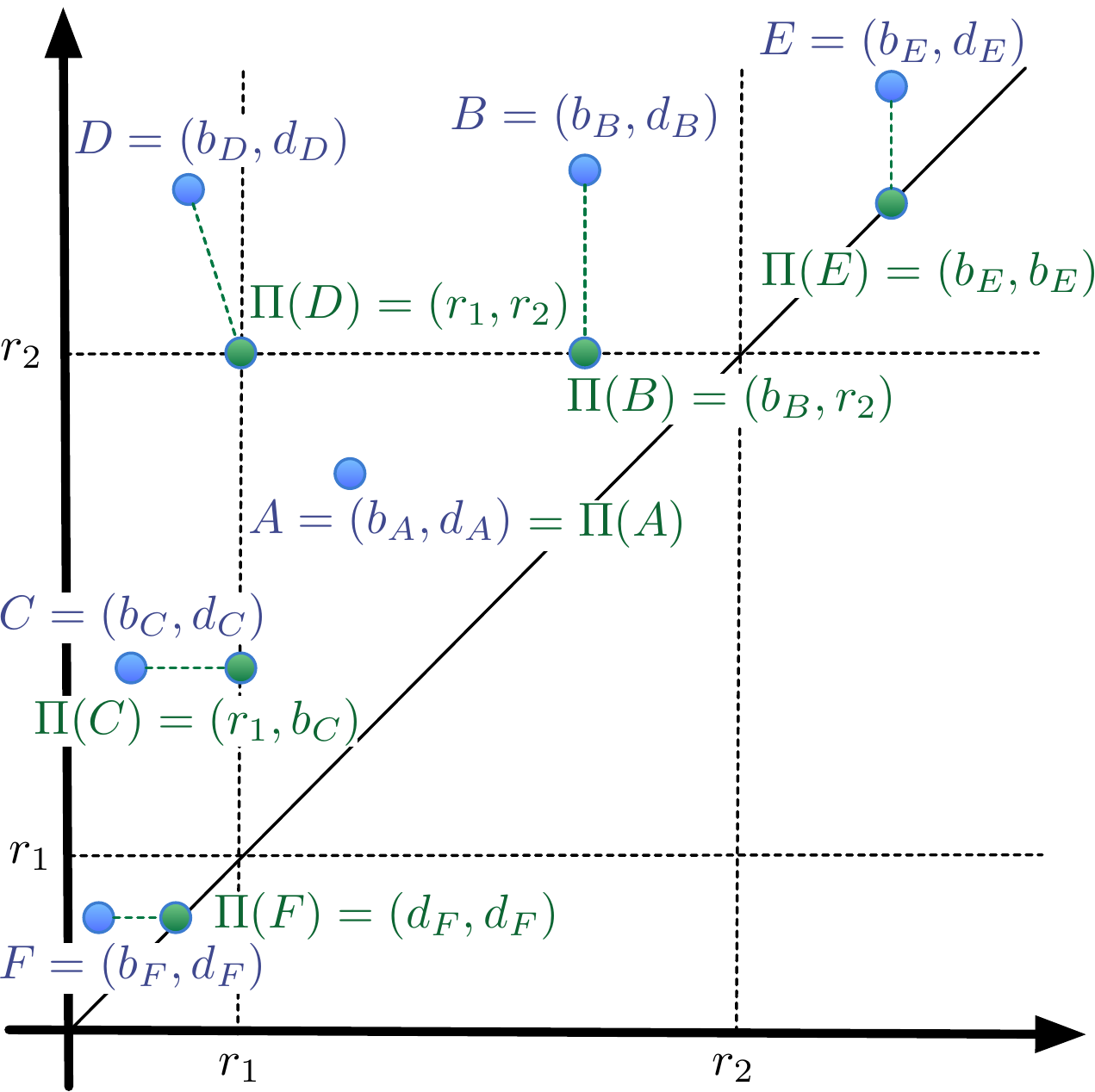}
\label{fig:projections}
}\end{defn}


Typically, a persistence diagram is considered to be a set of points $\{(b, d)\}$ for which $b < d$. 
In order to compute the bottleneck distance, one adds countably many copies of the diagonal $\{(x, x) : x \in \mathbb{R}\}$, which may intuitively correspond to topological features that are born and simultaneously die (and thus, never really exist at all). 
This allows for a point in one persistence diagram to be matched to the diagonal if it is far away from any point in the other diagram, and also accounts for the fact that two persistence diagrams may have different numbers of off-diagonal points. 
Notice that points like $E$ and $F$ in the above figure correspond to features that are born and die outside of the interval $I$ (either completely before or completely after). 
The restriction result cited above from \cite{CarlssonSilva2010}, defining $\Uspace[r_1, r_2]$, would not include points $\Pi(E)$ or $\Pi(F)$ in its diagram. 
But, since both $\Pi(E)$ and $\Pi(F)$ are on the diagonal, including them in $\DgPD \Xbf^I$ does not change the bottleneck distance between two restricted diagrams.



\section{Bottleneck Distance in the Local vs. Global Settings}
\label{sec:bottleneck}


\noindent In this section, we prove our main result relating the bottleneck distance between persistence diagrams with the bottleneck distance between their interval-restricted versions.


\begin{thm}\label{thm:changeradius}
Let $\Xspace_1 \leftrightarrow \Xspace_2 \leftrightarrow \ldots \leftrightarrow \Xspace_n$ and $\Yspace_1 \leftrightarrow \Yspace_2 \leftrightarrow \ldots \leftrightarrow \Yspace_n$ be two sequences of topological spaces and continuous maps, and let
      $\DgPD{\Xbf}$ and $\DgPD{\Ybf}$ be their corresponding zigzag persistence diagrams. Consider the interval $I=[r_1, r_2] \subset \Rspace$ and let $\DgPD{\Xbf^I}$ and $\DgPD{\Ybf^I}$ be the restrictions of these diagrams to $I$. Then $d_B(\DgPD{\Xbf^I},\DgPD{\Ybf^I}) \leq d_B(\DgPD{\Xbf},\DgPD{\Ybf}).$
\end{thm}
%

\begin{proof}

Let $\mu \subseteq \DgPD{\Xbf} \times \DgPD{\Ybf}$ be a partial matching. 
For computation of the bottleneck distance, we say that any unpaired point in one of the persistence diagrams is matched to the nearest point (in the $L_\infty$ norm) on the diagonal $\Delta=\{(x,x):x\in \Rspace\}$. 

Consider $\hat{\mu} \subseteq \DgPD{\Xbf^I} \times \DgPD{\Ybf^I}$ defined such that, for each $(p, q) \in \mu$, we have $(\Pi(p), \Pi(q)) \in \hat{\mu}$. We claim that this is a valid partial matching between the two restricted diagrams. A partial matching means that no two points $\hat{p}, \hat{p}' \in \DgPD{\Xbf^I}$ are matched to the same $\hat{q} \in \DgPD{\Ybf^I}$, and similarly the same $\hat{p} \in \DgPD{\Xbf^I}$ is not matched to two different points $\hat{q}, \hat{q}' \in \DgPD{\Ybf^I}$. We show the first case and remark that the second case is proved similarly. Assume, for the sake of contradiction, that $(\hat{p}, \hat{q}), (\hat{p}', \hat{q}) \in \hat{\mu}$. By definition of $\hat{\mu}$ we must have $(p, q), (p', q') \in \mu$ such that $\hat{p} = \Pi(p)$, $\hat{p}' = \Pi(p')$, and $\hat{q} = \Pi(q) = \Pi(q')$. Recall that persistence diagrams are multisets, so the fact that $\hat{q} = \Pi(q) = \Pi(q')$ would indicate that there are two copies of the point $\hat{q}$ in $\hat{\mu}$, and that $\hat{p}$ is matched to one copy and $\hat{p}'$ is matched to the other copy. The only case in which we wouldn't have two copies of $\hat{q}$ is if $q=q'$, but this would contradict $\mu$ being a partial matching\footnote{Of course we could have $q=q'$ if they have the same coordinates, but if there are multiple copies of the same $(b, d)$ point, we count them as different points in the persistence diagram, and thus not equal.}.

What is left to show is that the maximal distance between matched points $\hat{\mu}$ is less than that for $\mu$, a fact proved in the following lemma. Indeed, if $\mu$ is the matching that achieves the bottleneck distance between $\DgPD{\Xbf}$ and $\DgPD{\Ybf}$ and the cost of $\hat{\mu}$ is smaller, then the bottleneck distance between $\DgPD{\Xbf^I}$ and $\DgPD{\Ybf^I}$ will only be smaller still. 
\end{proof}

\begin{lem}\label{lem:suprestriction}
For the partial matching $\hat{\mu}$,
$\displaystyle \sup_{(\hat{p},\hat{q})\in\hat{\mu}} ||\hat{p}-\hat{q}||_\infty \leq \displaystyle \sup_{(p,q)\in\mu} ||p-q||_\infty.$
\end{lem}

\begin{proof}
Consider two points $p\in \DgPD{\Xbf}$ and $q\in \DgPD{\Ybf}$ achieving $\ds\sup_{(p,q)\in\mu} ||p-q||_\infty$. A case analysis of the 21 
possible pairings of points will establish the lemma. First, observe that if either $p$ or $q$ is in \textbf{Case A}, then after projecting onto the restricted region, at least one point is unchanged and at most one point is moved closer, yielding the desired inequality.

Next, we will consider the scenarios when one of the points, say (without loss of generality) $q=(b_q,d_q)$, belongs to \textbf{Case B}, so that $\Pi(q)=(b_q,r_2)$. If $p=(b_p,d_p)$ is also a \textbf{Case B} point, then the inequality holds because 
projecting the points does not change the horizontal distance and the vertical distance of the projection is 0.
In the case that $p=(b_p,d_p) \in \textbf{Case C}$, we have $\Pi(p)=(r_1,d_p)$ and $||\Pi(p)-\Pi(q)||_\infty=\max\{b_q-r_1,r_2-d_p\}$. Since $b_p \leq r_1 \leq d_p$ and $b_q \leq r_2 \leq d_q$, the horizontal distances satisfy $b_q-r_1 \leq b_q-b_p$ and the vertical distances satisfy $r_2-d_p \leq d_q - d_p$, so that the inequality holds. 
If $p=(b_p,d_p) \in \textbf{Case D}$, then $\Pi(p)=(r_1,r_2)$ and $||\Pi(p)-\Pi(q)||_\infty=b_q-r_1$, since the vertical distance between the projections is 0. This in turn is less than $b_q-b_p \leq ||p-q||_\infty$.
Now, if $p=(b_p,d_p) \in \textbf{Case E}$, we have $\Pi(p)=(b_p,b_p)$ and  $||\Pi(p)-\Pi(q)||_\infty=\max\{|b_p-b_q|,b_p-r_2\}$. Since $b_q \leq r_2\leq b_p$, this implies that the horizontal distance between the projections must be larger than the vertical distance. Therefore, $||\Pi(p)-\Pi(q)||_\infty=b_p-b_q \leq \max\{b_p-b_q,|d_p-d_q|\}=||p-q||_\infty$.
Finally, if $p=(b_p,d_p) \in \textbf{Case F}$, then $\Pi(p)=(d_p,d_p)$ and $||\Pi(p)-\Pi(q)||_\infty=\max\{b_q-d_p,r_2-d_p\}$. Since $b_p \leq d_p \leq r_1 \leq b_q \leq r_2 \leq d_q$, the horizontal distances satisfy $b_q-d_p \leq b_q-b_p$ and the vertical distances satisfy $r_2-d_p \leq d_q - d_p$, yielding the desired inequality.

The case analysis for the remaining pairings proceeds in a similar manner.
\end{proof}

Given an $\Rspace$-valued function, there is a natural construction of a {level set zigzag (LZZ) persistence module}~\cite{CarlssonSilvaMorozov2009} that sweeps its level sets from bottom to top~\cite{EdelsbrunnerMorozov2017}. 
Given a topological space $\Xspace$ and a continuous function $f : \Xspace \rightarrow \R$ of Morse type, let $\Xspace_t=f^{-1}(t)$ denote the \emph{level set} of $f$ for any $t\in\R$ and $\Xspace_I = f^{-1}(I)$ denote the \emph{slice} of $\Xspace$ which $f$ maps to the interval $I\subset \Rspace$. If $I=[a,b]$, we may denote this as $\Xspace_a^b$. 
Recall that $(\Xspace, f)$ is of \emph{Morse type} if, for the finite set of critical values $a_1<a_2<\ldots <a_n$ of $f$, the open intervals $(-\infty,a_1),(a_1,a_2),\ldots,(a_{n-1},a_n),(a_n,\infty)$ are such that for each interval $I$, $f^{-1}(I)$ is homeomorphic to $\Yspace\times I$ for some compact and locally connected space $\Yspace$ with $f$ serving as the projection onto $I$~\cite{CarlssonSilvaMorozov2009}. The homeomorphisms should extend to continuous functions on $\Yspace \times \bar{I}$, where $\bar{I}$ is the closure of $I$ in $\Rspace$, and each $\Xspace_t$ should also have finitely-generated homology. Then, given $(\Xspace, f)$ of Morse type with critical values $a_i$ as above, we choose arbitrary $s_i$ satisfying
\[ -\infty < s_0 < a_1 < s_1 < a_2 < \cdots < s_{n-1} < a_n < s_n < \infty.\]
The \emph{level set zigzag persistence} of $(\Xspace, f)$ is defined to be the zigzag persistence for the sequence
\[ \Xspace_{s_0}^{s_0} \rightarrow \Xspace_{s_0}^{s_1} \leftarrow \Xspace_{s_1}^{s_1} \rightarrow \Xspace_{s_1}^{s_2} \leftarrow \cdots \rightarrow \Xspace_{s_{n-1}}^{s_n} \leftarrow \Xspace_{s_n}^{s_n}.  \]
We denote the persistence diagram by $\DgPD{f}$. 

The level set zigzag persistence can be used to compute the ordinary persistent homology of an $\Rspace$-valued function with good space efficiency. 
In particular, the LZZ module is related to the ordinary (extended) persistence module via the Mayer-Vietoris pyramid~\cite[Figure 3]{CarlssonSilvaMorozov2009}, where the zigzag sequence and the ordinary sequence are shown to contain the same information in their persistent homology. 
Therefore, we could use the algorithm for zigzag persistent homology to compute extended persistence, while using space that depends only on the size of the largest level set instead of the entire domain~\cite{CarlssonSilvaMorozov2009, MilosavljevicMorozovSkraba2011}.

We now state a straightforward corollary to Theorem \ref{thm:changeradius} which we will use in Section \ref{sec:applications}.

\begin{cor}\label{cor:levelsetzigzag}
Let $f:\Xspace \to \Rspace$ and $g:\Yspace \to \Rspace$ be Morse type functions defined on topological spaces $\Xspace$ and $\Yspace$, and for an interval $I=[r_1,r_2]$, let $\DgPD{f^I}$ and $\DgPD{g^I}$ be the restrictions of the LZZ persistence diagrams $\DgPD{f}$ and $\DgPD{g}$ to the interval $I$. Then $d_B(\DgPD{f^I},\DgPD{g^I}) \leq d_B(\DgPD{f},\DgPD{g}).$
\end{cor}
\section{Applications to Metric Graphs and $d$-Parameter Persistence}
\label{sec:applications}

For uses of Corollary \ref{cor:levelsetzigzag}, we turn to the \emph{metric graph} setting. 
Metric graphs commonly arise when studying road networks as well as biological or chemical structure graphs.
Given a graph $G$ with a set of vertices and edges, a length function on the edges, and a geometric realization $|G|$ of the graph, one may specify a metric on $G$ by taking the minimum length of any path between any pair of points (not necessarily vertices) in the geometric realization.
Given a base point $v \in |G|$, the \emph{geodesic distance function} $f_v : |G| \rightarrow \R$ is given by $f_v(x) = d_G(v,x)$. Then $\Dg{f_v}$ denotes the $0$-dimensional LZZ persistence diagram induced by $f_v$. Equivalently, $\Dg{f_v}$ is the union of the $0$- and $1$-dimensional \emph{extended persistence diagrams} for $f_v$ (see \cite{cohen2009extending} for the details of extended persistence).
Corollary \ref{cor:levelsetzigzag} can be used to compare local neighborhoods of two different metric graphs, $G_1$ and $G_2$, with base points $v \in G_1$ and $u \in G_2$. 
In particular, given $f_v : |G_1| \rightarrow \mathbb{R}$ and $g_u : |G_2| \rightarrow \mathbb{R}$, we have $d_B(\DgPD{f_v}^I,\DgPD{g_u}^I) \leq d_B(\DgPD{f_v},\DgPD{g_u})$ for any real interval $I$. 
Typically, for comparing local neighborhoods, $I=[0, r]$.
The following corollary gives a stability-type result for comparing two local neighborhoods within a single metric graph. 

\begin{cor}\label{cor:changebasepoint}
Let $G$ be a metric graph with geometric realization $|G|$. For a fixed interval $I$ and points $u,v\in|G|$, we have $d_B(\DgPD{f_u^I},\DgPD{f_v^I}) \leq d_G(u,v)$.
\end{cor}
\begin{proof}
By Corollary \ref{cor:levelsetzigzag}, $d_B(\DgPD{f_u}^I,\DgPD{f_v}^I) \leq d_B(\DgPD{f_u},\DgPD{f_v})$. Since $f_u,f_v : |G| \to \mathbb{R}$ are two Morse type functions, $d_B(\DgPD{f_u},\DgPD{f_v})\leq || f_u-f_v ||_\infty$ by the LZZ  Stability Theorem of \cite{CarlssonSilvaMorozov2009}. 
Furthermore, by the triangle inequality, 
for any $x\in |G|$, $|d_G(x, u) - d_G(x, v)| \le d_G(u, v)$, meaning that $||f_u-f_v||_\infty \leq d_G(u,v)$. 
Putting everything together proves the claim. 
\end{proof}

Another application of Corollary \ref{cor:levelsetzigzag} is as follows.
Define $\Phi : |G| \rightarrow SpDg$, $\Phi(v) = \DgPD{f_v},$ where $SpDg$ denotes the space of persistence diagrams. 
Given metric graphs $(G_1, d_{G_1})$ and $(G_2, d_{G_2})$, their \emph{persistence distortion distance} \cite{DeyShiWang2015} is
\[ d_{PD}(G_1, G_2) := d_H(\Phi(|G_1|), \Phi(|G_2|)),\]
where $d_H$ denotes the Hausdorff distance. 
In other words,
\[ 
d_{PD}(G_1, G_2) = \max \left\{ \sup_{D_1 \in \Phi(|G_1|)} \infd_{D_2 \in \Phi(|G_2|)} d_B(D_1, D_2),  \sup_{D_2 \in \Phi(|G_2|)} \infd_{D_1 \in \Phi(|G_1|)} d_B(D_1, D_2)\right\}. 
\]
Note that the diagram $\DgPD{f_v}$ contains both $0$-  and $1$-dimensional persistence points, but only points of the same dimension are matched under the bottleneck distance.
A local version of the persistence distortion distance, which we will denote by $d_{PD}^r$, may be defined as follows: for each base point $v$, only consider the distance function to points within a fixed intrinsic radius $r$.

 \begin{cor}
 \label{cor:localpd}
 If $r \leq r'$, then $d_{PD}^r(G_1, G_2) \leq d_{PD}^{r'}(G_1, G_2).$
 \end{cor}
 
 \begin{proof} Let $D_1^{r}$ be the persistence diagram for some base point $v \in |G_1|$,  where the geodesic distance function is computed in the interval $[0,r]$.  Let $D_1^{r'}$ be the persistence diagram for the same base point, but where the distance function is computed in the interval $[0,r']$.  Define $D_2^{r}$ and $D_2^{r'}$ similarly for some base point in $|G_2|$.  By viewing $D_i^{r}$ as a restriction of $D_i^{r'}$ for $i=1,2$, we can apply Theorem \ref{thm:changeradius} to show that $d_B(D_1^{r},D_2^{r}) \leq d_B(D_1^{r'},D_2^{r'})$.  Since our choice of base points was arbitrary, this inequality holds for persistence diagrams across all choices of base points in $|G_1|$ and $|G_2|$.  Therefore, using the definition of the local version of the persistence distortion distance, we can conclude that $d_{PD}^r(G_1, G_2) \leq d_{PD}^{r'}(G_1, G_2).$
\end{proof}

%




We end with a final remark on how Theorem~\ref{thm:changeradius} can be applied to a $d$-parameter persistence module on any topological space (not restricted to the level set or metric graph settings). 
A $d$-parameter persistence module is indexed by a $d$-dimensional family of vector spaces, $\{\Vspace_u\}_{u\in\Rspace^d}$, together with a family of linear maps $\{\varphi_\Vspace(u,v):\Vspace_u \to \Vspace_v\}_{u\preceq v}$ such that for $u\preceq v \preceq w \in \Rspace^d$, we have $\varphi_\Vspace(u,u)=\id_{\Vspace_u}$ and $\varphi_\Vspace(u,w) \circ \varphi_\Vspace(u,v)=\varphi_\Vspace(u,w)$ \cite{Carlsson2009}. Here, $u\preceq v$ if and only if $u_i \leq v_i$ for $i=1,\ldots,d$. 
Any line $L$ in the set of all lines of $\mathbb{R}^d$ with direction $\textbf{m}=(m_1,\ldots,m_d)$ such that $\ds\min_i m_i$ is strictly positive gives a one-parameter slice of the $d$-parameter persistence module. 
Given two $d$-parameter persistence modules $\Xbf$ and $\Ybf$, we define their \emph{matching distance} \cite{Landi2018} to be
\[d_{match}(\Xbf, \Ybf):=\displaystyle \sup_{L} \displaystyle \min_i m_i d_B(\Dg \Xbf_{L}, \Dg \Ybf_{L}),\]
where $\Dg \Xbf_{L}$ and $\Dg \Ybf_{L}$ are the persistence diagrams of the $d$-parameter persistence modules $\Xbf$ and $\Ybf$ restricted along line $L$.  
Our result extends naturally to this linear relationship between these two parameters. 
Indeed, if we restrict both $d$-parameter persistence modules to a region $\Ispace=I_1 \times \cdots \times I_d$, where each $I_i$ is an interval of the real line, then Theorem \ref{thm:changeradius} implies the following corollary.
\begin{cor}
\[ d_{match}(\Xbf^\Ispace, \Ybf^\Ispace) \leq d_{match}(\Xbf, \Ybf)\]
where $d_{match}(\Xbf^\Ispace, \Ybf^\Ispace)$ is computed by restricting $\Dg{\Xbf_{L}}$ and $\Dg{\Ybf_{L}}$ to the subinterval of the line $L$ passing through the region $\Ispace$.
\end{cor}
\begin{proof}
For a fixed line $L$ with direction $\textbf{m}$, consider a region $\Ispace$ restricted to $L$, denoted $I_L \subset \Ispace \cap L$. 
Recall that $\Dg \Xbf_{L}$ and $\Dg \Ybf_{L}$ are the persistence diagrams of the $d$-parameter persistence modules $\Xbf$ and $\Ybf$ restricted along the line $L$.  
Based on Theorem \ref{thm:changeradius},
\begin{align}\label{eqn:match_one_line}
d_B(\DgPD{\Xbf_L^{I_L}},\DgPD{\Ybf_L^{I_L}}) \leq d_B(\DgPD{\Xbf_L},\DgPD{\Ybf_L}).
\end{align}
From the definition of supremum, we know that $\forall \epsilon > 0$, there is a line $L_\epsilon$ such that
\[ d_{match}(\Xbf^\Ispace, \Ybf^\Ispace) -\epsilon < \min_i m_i d_B(\Dg \Xbf_{L_\epsilon}^{I_{L_\epsilon}}, \Dg \Ybf_{L_\epsilon}^{I_{L_\epsilon}}). \]
Using observation \eqref{eqn:match_one_line} above, we see that 
\[ d_{match}(\Xbf^\Ispace, \Ybf^\Ispace) -\epsilon < \min_i m_i d_B(\DgPD{\Xbf_{L_\epsilon}},\DgPD{\Ybf_{L_\epsilon}}). \]
The right-hand side is, of course, less than the supremum over all lines $L$, the definition of $d_{match}(\Xbf, \Ybf)$.
Hence, for every $\epsilon > 0$, we have
$d_{match}(\Xbf^\Ispace, \Ybf^\Ispace) -\epsilon < d_{match}(\Xbf, \Ybf)$; in other words,
$d_{match}(\Xbf^\Ispace, \Ybf^\Ispace) \leq d_{match}(\Xbf, \Ybf)$,
as desired.
\end{proof}


\section{Discussion}
\label{sec:conclusion}


Theorem \ref{thm:changeradius} and its corollaries provide explicit relationships between distances computed locally and globally, and the resulting inequalities are very broadly applicable.
For instance, the fact that the local bottleneck distance is bounded above by the global bottleneck distance allows for a single global computation to potentially rule out local differences if the global distance is low.
If looking for local differences, starting with a global computation may save computational time if there are too many local comparisons to make.
On the other hand, the global bottleneck distance being bounded below by the local version allows smaller computations to approach the global truth, while perhaps being more computationally tractable.

In future work, we would like to extend these ideas to generalized persistence, where instead of a linear sequence of topological spaces one considers topological spaces and transformations that form a poset. 
In contrast to zigzag persistence, this generalized persistence does not have the notion of a persistence diagram.
Instead, we would need to restate our results in terms of the interleaving distance between persistence modules.
Moreover, a notion of ``local'' would have to be defined in the poset setting.\\

\bibliographystyle{plain}
\bibliography{local-global}

\end{document}